\documentclass[letterpaper,12pt]{article}
\usepackage[utf8]{inputenc}
\usepackage[english]{babel}
\usepackage{verbatim,color}

\usepackage{amsfonts,amsmath,amssymb,tikz}
\usepackage{amsthm}

\newcommand{\cp}{\mathbin{\, \Box \,}}
\newcommand{\cF}{\mathcal{F}}

\tikzstyle{vertex}=[circle, draw, inner sep=0pt, minimum size=6pt]
\newtheorem{theorem}{Theorem}
\newtheorem{lemma}{Lemma}
\newtheorem{corollary}{Corollary}
\newtheorem{obs}{Observation}
\newtheorem{prop}{Proposition}

\begin{document}

\title{A characterization of well-dominated\\ Cartesian products}

\author{
$^{a}$Kirsti Kuenzel
\and
$^{b}$Douglas F. Rall
}

%\date{\today}

\maketitle

\begin{center}
$^a$ Department of Mathematics, Trinity College, Hartford, CT\\
$^b$ Department of Mathematics, Furman University, Greenville, SC\\

\end{center}

\begin{abstract}
A graph is well-dominated if all its minimal dominating sets have the same cardinality.  In this paper we prove that at least one factor of every connected, well-dominated
Cartesian product is a complete graph, which then allows us to give a complete characterization of the connected, well-dominated Cartesian products if both factors
have order at least $2$. In particular, we show that $G\,\Box\,H$ is well-dominated if and only if $G\,\Box\,H = P_3 \,\Box\,K_3$  or $G\,\Box\,H= K_n \,\Box\,K_n$ for some $n\ge 2$.
\end{abstract}



\section{Introduction}
A minimal dominating set in any finite graph can be found in linear time by processing its vertices sequentially and retaining only those vertices that are needed
to dominate the graph.   The invariant of interest in most applications is the domination number of the graph, and the related decision problem is a well-known NP-complete problem~\cite{gj-79}.  The class of \emph{well-dominated} graphs, introduced by Finbow, Hartnell, and Nowakowski~\cite{fhn-1988}, are those for which every minimal
dominating set has the same cardinality.  Thus, a graph is well-dominated if and only if the simple algorithm described above always produces a dominating set of
minimum cardinality.  The structure of well-dominated graphs is far from being known, but the collection of connected, well-dominated graphs within several classes of graphs have
recently been determined. This list includes bipartite~\cite{fhn-1988}, girth larger than $4$~\cite{fhn-1988}, simplicial and chordal~\cite{ptv-1996}, and graphs
without cycles of length $4$ or $5$~\cite{lt-2017}.  In addition, a characterization has been given of the well-dominated graphs within the classes of direct products~\cite{r-2023}, lexicographic products~\cite{ghm-2017}, and disjunctive products~\cite{akr-2021}.  Several subclasses of well-dominated strong products and Cartesian products have also been exhibited.
%%% ************************

In the initial attempt to determine the well-dominated Cartesian products, Anderson, Kuenzel and Rall~\cite{akr-2021} characterized the nontrivial, well-dominated
Cartesian products when both factors are triangle-free.
\begin{theorem}~{\rm (\cite[Theorem 2]{akr-2021})} \label{thm:maincp}
Let $G$ and $H$ be nontrivial, connected graphs both of which have girth at least $4$. The Cartesian product  $G\,\Box\, H$ is well-dominated  if and only if $G = H = K_2$.
\end{theorem}

\noindent The same paper by Anderson, et al. included the following result.
\begin{theorem}~{\rm (\cite[Theorem 1]{akr-2021})} \label{thm:wdCartesian}
Let $G$ and $H$ be connected graphs. If the Cartesian product $G \cp H$ is well-dominated, then $G$ or $H$ is well-dominated.
\end{theorem}
\noindent However, in a private communication, Erika King and Michael O’Grady~\cite{ko-2023} pointed out a logical gap in the proof of Theorem~\ref{thm:wdCartesian}.

Rall~\cite{r-2023} determined the connected, well-dominated Cartesian products when one of the factors is a complete graph of order at least $2$.
In particular, he proved the following result.

\begin{theorem}~{\rm (\cite[Theorem 3]{r-2023})}\label{thm:wdCart}
Let $m$ be a positive integer with $m\ge 2$ and let $H$ be a nontrivial, connected graph. The Cartesian product $K_m\, \Box\, H$ is well-dominated
if and only if either $m \ne 3$ and $H= K_m$ or $m=3$ and $H \in \{P_3, K_3\}$.
\end{theorem}
\noindent In addition, he conjectured that every nontrivial, connected and well-dominated Cartesian product has a complete graph as a factor.
Our main result in this paper is the following theorem that proves this conjecture and thus verifies Theorem~\ref{thm:wdCartesian}.

\begin{theorem} \label{thm:CartesianWD}
If $G\,\Box\, H$ is connected and well-dominated, then  $G$ or $H$ is a complete graph.
\end{theorem}

\noindent Applied together with Theorem~\ref{thm:wdCart}, this  completes the characterization of the class of nontrivial, connected, well-dominated Cartesian products.

In Section~\ref{sec:Defns} we supply necessary definitions and provide additional background results in the study of well-dominated Cartesian products.  In addition,
we prove several preliminary lemmas that will be used in our proof of the main theorem.  Section~\ref{sec:completefactor} is devoted to proving Theorem~\ref{thm:CartesianWD}.

\section{Definitions and Preliminary Results} \label{sec:Defns}

In this paper we restrict our attention to finite, simple graphs.  We say a graph $G$ is \emph{nontrivial} if its order, denoted by $n(G)$, is at least $2$, and a Cartesian product
is said to be \emph{nontrivial} if both factors are nontrivial.
In general we follow the definitions and notation of the book by Haynes, Hedetniemi and Henning~\cite{hhh-2023}.  For a positive integer $n$, the set
$\{1,\ldots,n\}$ will be denoted by $[n]$.  If $g$ is a vertex in a graph $G$, the set of vertices adjacent to $g$ is denoted by $N_G(g)$ and is called the \emph{open neighborhood}
of $g$.  The \emph{closed neighborhood} of $g$ is the set $N_G[g]$ defined by $N_G[g]=N_G(g) \cup \{g\}$.  For a set $S \subseteq V(G)$, the open neighborhood of $S$
is defined by $N_G(S)=\cup_{g\in S}N_G(g)$ and its closed neighborhood is the set $N_G[S]=N_G(S)\cup S$. We will omit the subscript on these neighborhood set
names if the graph is clear from the context.  A \emph{dominating set} of $G$ is a subset $S$ of the vertex set of
$G$ such that $N[S]=V(G)$.  A dominating set $S$ is minimal if $S-\{x\}$ does not dominate $G$ for every $x \in S$.  This condition is equivalent to requiring
$N[x] - N[S-\{x\}] \neq \emptyset$ for every $x\in S$.  When $N[u] - N[A-\{u\}] \neq \emptyset$ for a vertex $u \in A$, we say that $u$ has a \emph{private
neighbor with respect to $A$}.  When the set $A$ is clear from the context we may simply say that $u$ has a private neighbor.
If this is the case for each $u$ in $A$, then we say that $A$ is \emph{irredundant}.  If $A$ is irredundant, then some of the private neighbors
of a vertex $u \in A$ might belong to $A$ (when $N(u) \cap A=\emptyset$) and others might belong to $V(G)-A$.  If $A$ is an irredundant set and every vertex $u \in A$
has a private neighbor with respect to $A$ that does not belong to $A$, then we say $A$ is \emph{open irredundant}.  A set $I \subseteq V(G)$ is \emph{independent} if
no pair of distinct vertices in $I$ are adjacent.

The \emph{domination number} of $G$ is the cardinality of a smallest dominating set of $G$ and is denoted $\gamma(G)$.  A dominating set $D$ of $G$ with $|D|=\gamma(G)$ is
called a $\gamma(G)$-set.  The \emph{vertex independence number} of $G$ is the cardinality of a largest independent set in $G$ and is denoted $\alpha(G)$.  A set $I \subseteq V(G)$ is called a $\alpha(G)$-set if $I$ is independent and has cardinality $\alpha(G)$.  The graph $G$ is said to be \emph{well-dominated}
if all of its minimal dominating sets have cardinality $\gamma(G)$, and $G$ is \emph{well-covered} if all of its maximal independent sets have cardinality $\alpha(G)$.
One can greedily add vertices to any independent set in $G$ to enlarge it to a maximal independent set.  It follows that in a well-covered graph
$G$ any independent set is a subset of an $\alpha(G)$-set.  It is clear from the definition that a maximal independent set in $G$ is a minimal dominating set of $G$.
Consequently, a graph that is well-dominated is also well-covered.

If $A$ is a nonempty set, then a collection of pairwise disjoint subsets of $A$ whose union is $A$ is called a \emph{weak partition} of $A$.  Let $X$ and $Y$ be graphs.  The
\emph{Cartesian product}, $X\,\Box \,Y$, is the graph whose vertex set is $V(X) \times V(Y)$.  Two vertices $(x_1,y_1)$ and $(x_2,y_2)$ are adjacent in $X\,\Box \,Y$
if they are equal in one coordinate and adjacent in the other coordinate.  The Cartesian product is associative, commutative and distributes over disjoint
unions.  For a fixed vertex $x\in V(X)$, the set $\{(x,y) \,|\, y \in V(Y)\}$ is called a \emph{$Y$-layer}.
Each $Y$-layer induces a subgraph of $X\,\Box \,Y$ that is isomorphic to $Y$.  Similarly, for a fixed vertex $y$ in $V(Y)$, the set
$\{(x,y) \,|\, x \in V(X)\}$ is called a \emph{$X$-layer}, and it induces a subgraph isomorphic to $X$.

We now list some preliminary results that will prove useful in the remainder of the paper.  Note that if $I$ is an independent set in a
graph $X$, then $I \cup D$ is a minimal dominating set of $X$ for every minimal dominating set $D$ of $X-N[I]$.  This implies the following observation
as noted by Finbow, Nowakowski and Hartnell~\cite{fhn-1988}.
\begin{obs}~{\rm (\cite{fhn-1988})} \label{obs:prelimfacts}
If $X$ is a well-dominated graph and $I$ is an independent set of $X$, then $X-N[I]$ is well-dominated.
\end{obs}
It is clear that the Cartesian product of two connected graphs is itself connected.  Also, since the Cartesian product distributes over
disjoint unions, we get the following useful fact.
\begin{obs} \label{obs:connectedCP}
If $X \cp Y$ is connected, then $X$ and $Y$ are connected.
\end{obs}

\noindent The following observation follows directly from the definitions.
\begin{obs} \label{obs:components}
A graph is well-dominated if and only if each of its components is well-dominated.
\end{obs}

Although the well-covered Cartesian products have not been characterized, Hartnell and Rall reduced the characterization in the following result.

\begin{theorem}~{\rm (\cite{hr-2013})}\label{thm:wc-cartesian}
If $G$ and $H$ are graphs such that $G\cp H$ is well-covered, then at least one of $G$ or $H$ is well-covered.
\end{theorem}

\noindent Since a graph is well-covered if it is well-dominated, we will use Theorem~\ref{thm:wc-cartesian} to infer that at least one of the factors of a well-dominated Cartesian product is well-covered.

The following lemma established by Hartnell, Rall and Wash will be used to show that the Cartesian product of two triangle-free connected graphs
is not well-dominated if both have order at least $3$.
\begin{lemma}~{\rm (\cite[Lemma 4]{hrw-2018})}\label{lem:wcgirth4}
If $G$ and $H$ are connected graphs both having order at least $3$ and girth at least $4$, then $G\, \Box\, H$ is not well-covered.
\end{lemma}

Bollob\'{a}s and Cockayne~\cite{bc-1979} showed that every connected graph has a minimum dominating set that is open irredundant.  Using this fact, Rall
proved the following proposition by showing that $D \times V(Y)$ is a minimal dominating set of $X \, \Box\, Y$ for any graph $Y$ if $D$ is
an open irredundant dominating set of $X$.

\begin{prop}~{\rm (\cite[Proposition 16]{r-2023})}\label{prop:domfactors}
Let $X$ and $Y$ be nontrivial connected graphs. If $X\,\Box\, Y$ is well-dominated, then $\gamma(X\,\Box\, Y) = \gamma(X)n(Y) = \gamma(Y)n(X)$.
\end{prop}

The method we will use to prove  Theorem~\ref{thm:CartesianWD} involves assuming that $G\,\Box\,H$ is a minimal counterexample.  By removing the
closed neighborhood of a maximal independent set in $G\,\Box\,H$, we then will analyze the Cartesian product of smaller factors that may be disconnected.
To determine the structure of these product graphs we need the following three lemmas.

\begin{lemma} \label{lem:path3}
Let $X$ be a connected graph of order at least $3$.  If $P_3\, \Box\, X$ is well-dominated, then $X=K_3$.
\end{lemma}
\begin{proof}
Let $V(P_3)=\{a,b,c\}$ with $\deg(b)=2$, and suppose $P_3\, \Box\, X$ is  connected and well-dominated.  It is clear that the set $\{b\} \times V(X)$ is a minimal dominating set of $P_3\, \Box\, X$, and thus $\gamma(P_3\, \Box\, X)=n(X)$.  Since $P_3\, \Box\, X$ is well-covered, we infer that $X$ contains a triangle by applying Lemma~\ref{lem:wcgirth4}.  Suppose first that $X$ contains a vertex $x$ of degree at least $3$.  Let
\[ S=(\{a,b,c\} \times \{x\}) \cup (\{b\} \times (V(X)-N_X[x]))\,.\]
It is easy to see that $S$ dominates $P_3\, \Box\, X$.  Also, $|S|=3+n(X)-(\deg(x)+1)<n(X)$.  This is a contradiction.  Therefore, $\Delta(X) \le 2$, and
this implies that $X=K_3$.
\end{proof}

\begin{lemma}\label{lem:bipartite} Let $X$ be a connected graph of order at least $3$ and let $r$ and $s$ be positive integers.  If $2\le r \le s$ or if $r=1$ and $s \ge 3$, then
$K_{r, s} \,\Box\, X$ is not well-dominated.
\end{lemma}

\begin{proof}
Suppose to the contrary that there exists some connected graph $X$ with order at least $3$ such that $K_{r, s} \,\Box\, X$ is well-dominated. Label the vertices of the two partite sets of $K_{r, s}$ as $\{a_1, \dots, a_r\}$ and $\{b_1, \dots, b_s\}$.

Assume first that $ 2 \le r \le s$. Let $D_X$ be any minimal dominating set of $X$.  Since $X$ is connected and has order at least $3$,
it follows that $V(X)-D_X$ is also a dominating set of $X$.  We claim that
\[D_1 = (\{a_1\} \times (V(X) - D_X)) \cup (\{a_2, \dots, a_r\} \times D_X)\]
is a minimal dominating set of $K_{r, s}\,\Box\, X$. Let $(u,x) \in V(K_{r, s}\,\Box\, X)-D_1$.  If $x \in V(X) - D_X$ and $u \not\in \{a_2, \dots, a_r\}$,
then $(u, x)$ is dominated by $(a_1, x)$.   If $x \in V(X) - D_X$ and $u = a_i$ for $2 \le i \le r$, then $(u, x)$ is dominated by $(a_i, x')$ for some $x' \in D_X$.
If $x \in D_X$ and $u \ne a_1$, then $(u, x)$ is dominated by $(a_2, x)$.  If $x \in D_X$ and $u = a_1$, then $(u, x)$ is dominated by  $(a_1, x')$
for some $x' \in V(X) - D_X$ since $V(X)-D_X$ is a dominating set of $X$.   Thus, $D_1$ is in fact a dominating set of $K_{r, s}\,\Box\, X$.
To show that $D_1$ is a minimal dominating set of $K_{r, s} \,\Box\, X$ we will show that each vertex in $D_1$ has a private neighbor with respect to $D_1$.
For any $x \in V(X) - D_X$, the vertex $(b_1, x)$ is a private neighbor of $(a_1, x)$ with respect to $D_1$.  Each vertex of $\{a_2, \dots, a_r\} \times D_X$ has a private neighbor with respect to $D_1$ in its $X$-layer since $D_X$ is a minimal dominating set of $X$ and $\{a_1, \dots, a_r\}$ is an independent set in $K_{r, s}$.
Therefore, $D_1$ is a minimal dominating set of $K_{r, s} \,\Box\, X$ of cardinality $n(X)+(r-2)|D_X|$.

One can easily verify that $D_3 = \{a_1, b_1\} \times V(X)$ is a minimal dominating set of $K_{r, s}\,\Box\, X$.  Since $K_{r, s}\,\Box\, X$
is well-dominated, it follows that
\[n(X)+(r-2)|D_X|=|D_1|=|D_3|=2n(X)\,,\]
and hence $(r-2)|D_X|=n(X)$.  Since $D_X$ is an arbitrary minimal dominating set of $X$, this implies that $X$ is well-dominated and $\gamma(X)=\frac{n(X)}{r-2}$.
By Proposition~\ref{prop:domfactors},
\[\gamma(K_{r, s}\,\Box\, X)=\gamma(K_{r,s})n(X)=\gamma(X)n(K_{r,s})=\frac{r+s}{r-2}n(X)\,,\]
which implies that $r=s+4$.  This contradiction shows that if $2\le r \le s$, then $K_{r, s} \,\Box\, X$ is not well-dominated.

Finally, suppose $r=1$ and $s \ge 3$ and let $D_X$ be an $\alpha(X)$-set.  Using arguments similar to those above, it follows that both
$\{a_1\} \times V(X)$ and $(\{b_1\} \times (V(X) -  D_X)) \cup (\{b_2, \dots, b_s\} \times D_X)$ are minimal dominating sets of
$K_{1, s} \,\Box\, X$. Therefore,
\[n(X) = n(X) - |D_X| + (s-1)|D_X|,\]
which implies $s=2$, another contradiction.
\end{proof}

Finally, we identify one more class of Cartesian products that are not well-dominated.
Let $\cF_1$ denote the class of all graphs of order at least $4$ obtained by attaching some finite number of leaves to the vertices of a complete graph of order $3$.
Let $\cF_2$ be the class of all graphs of order at least $5$ constructed from a complete graph of order $4$ by attaching some finite number of leaves to at most three of its vertices.

\begin{lemma}\label{lem:spwd} Let $F_1 \in \cF_1$ and let $F_2 \in \cF_2$.  If $X$ is any connected graph of order $3$ or more,
then neither $F_1 \,\Box\, X$ nor $F_2 \,\Box\, X$ is  well-dominated.
\end{lemma}

\begin{proof}
Suppose for the sake of obtaining a contradiction that $F_1 \,\Box\, X$ is well-dominated.  Let $y_1,y_2$ and $y_3$ be the vertices of $F_1$ that have degree at least $2$.
We let $L_i$ represent the set of leaves adjacent to $y_i$ for $i \in [3]$ and let $D_X$ be a  minimum dominating set of $X$. By the definition of $\cF_1$, at
least one of $y_1,y_2,y_3$ is adjacent to a leaf.  The argument is divided into three parts depending on how many of the sets $L_1,L_2,L_3$ are nonempty.

Suppose first that $L_i \ne \emptyset$ for every $i \in [3]$. Note that $\{y_1, y_2, y_3\}$ is an open irredundant dominating set of $F_1$ and therefore $D_1 = \{y_1, y_2, y_3\} \times V(X)$ is a minimal dominating set of $F_1\,\Box\, X$. Next, consider $D_2 = (\{y_1\} \times V(X)) \cup ((L_2 \cup L_3)\times D_X)$. We claim that $D_2$ is a minimal dominating set of $F_1 \,\Box\, X$. To see this, note that $(g, h)$ is dominated by $(y_1, h)$ if $h \in V(X)$ and $g \not\in L_2 \cup L_3$. If $g \in L_2 \cup L_3$, then $(g, h)$ is dominated by some $(g, h')$ where $h' \in D_X$. Moreover, $(x, h)$ is a private neighbor of $(y_1, h)$ where $x \in L_1$ and each vertex of $(L_2 \cup L_3) \times D_X$ has a private neighbor in its $X$-layer since $L_2 \cup L_3$ is an independent set in $F_1$.
Since $F_1 \,\Box\, X$ is well-dominated, it follows that
\[3n(X) = |D_1| = |D_2| = n(X) + (|L_2| + |L_3|)|D_X|.\]
Furthermore, $D_2' = (\{y_2\} \times V(X)) \cup ((L_1 \cup L_3)\times D_X)$ and $D_2'' = (\{y_3\} \times V(X)) \cup ((L_1 \cup L_2)\times D_X)$ are also  minimal dominating sets of $F_1 \,\Box\, X$. Thus, $|L_1| = |L_2| = |L_3|$. Therefore, $3n(X) = n(X) + 2|L_1||D_X| = n(X) + 2|L_1|\gamma(X)$ from which it follows that $n(X)/\gamma(X) = |L_1|$. On the other hand, since $\gamma(F_1)=3$, it follows
by Proposition~\ref{prop:domfactors} that  $\gamma(X)n(F_1)=\gamma(F_1)n(X) =3n(X)$, or equivalently, $n(X)/\gamma(X) = n(F_1)/3$. Hence,
\[|L_1| = \frac{n(X)}{\gamma(X)} = \frac{n(F_1)}{3} = |L_1| + 1,\]
which is clearly a contradiction. Therefore, this case cannot occur.

Next, suppose that $L_1 = \emptyset$ and $L_i \ne \emptyset$ for $i \in \{2, 3\}$. Note that $A_1 = (\{y_2\} \times V(X) ) \cup (L_3 \times D_X)$,
$A_2 = (\{y_1\} \times (V(X) - D_X)) \cup ((L_2 \cup L_3) \times D_X)$, and $A_3 = (\{y_3\} \times V(X) ) \cup (L_2 \times D_X)$ are all minimal dominating sets of $F_1 \,\Box\, X$. Therefore, using calculations similar to the first case above, we get
$|A_1| = |A_2|$, meaning $|L_2| = 1$ and similarly, $|L_3| = 1$. However, $\{y_2, y_3\} \times V(X)$ is also a minimal dominating set and now
\[n(X) + |D_X| = |A_1| = 2n(X)\] which cannot be for any nontrivial connected graph $X$.

Finally, assume $L_1 = L_2 = \emptyset$ and $L_3 \ne \emptyset$. As above, we know $\{y_3\} \times V(X)$ and
$(\{y_2\} \times V(X)) \cup (L_3 \times D_X)$ are minimal dominating sets of $F_1 \,\Box\, X$. It follows that
\[n(X) = n(X)  + |L_3||D_X|\,,\]
which implies $|L_3| =0$, another contradiction.  Therefore, $F_1 \,\Box\, X$ is not well-dominated.

Now suppose there exists a connected graph $X$ or order at least $3$ such that $F_2 \,\Box\, X$ is well-dominated.
Next, let $\{y_1,y_2,y_3,y_4\}$ induce a complete subgraph of $F_2$ and let $\deg(y_4)=3$.
Let $L_i$ be the set of leaves adjacent to $y_i$ for $i \in [3]$ and let $D_X$ be a minimum dominating set of $X$.

Suppose first that $L_i \ne \emptyset$ for $i \in [3]$. Note that $\{y_1, y_2, y_3\}$ is an open irredundant set of $F_2$ and therefore $D_1 = \{y_1, y_2, y_3\} \times V(X)$ is a minimal dominating set of $F_2\,\Box\, X$. By an argument similar to the first case above when we were analyzing $F_1 \,\Box\, X$, it follows that
$|L_1| = |L_2| = |L_3|$, $n(X)/\gamma(X) = |L_1|$, and $n(X)/\gamma(X) = n(F_2)/3$.  Hence,
\[|L_1| = \frac{n(X)}{\gamma(X)} = \frac{n(F_2)}{3} =\frac{4+3|L_1|}{3} = |L_1| + \frac{4}{3},\]
which is clearly a contradiction. Therefore, this case cannot occur.

Next, suppose that $L_1 = \emptyset$ and $L_i \ne \emptyset$ for $i \in \{2, 3\}$.  Note that
$(\{y_1\} \times V(X)) \cup ((L_2 \cup L_3) \times D_X)$ and $(\{y_3\} \times V(X) ) \cup (L_2 \times D_X)$ are both minimal dominating sets of
$F_2\,\Box\, X$.  This implies that $|L_3| = 0$, which is a contradiction.

Therefore, we may assume $L_1 = L_2 = \emptyset$ and $L_3 \ne \emptyset$.
It follows that $\{y_3\} \times V(X)$ and $(\{y_2\} \times V(X)) \cup (L_3 \times D_X)$ are minimal dominating sets of $F_2 \,\Box\, X$, which implies $L_3 = \emptyset$.
This final contradiction completes the proof.
\end{proof}

\section{The Characterization} \label{sec:completefactor}

In this section we prove Theorem~\ref{thm:CartesianWD}, which is restated below for the reader's convenience.
A fact that we will use a number of times is the following.

\begin{obs} \label{obs:wdreduction}
If $I_G$ is a maximal independent set in $G$ and $I_H$ is a maximal independent set in $H$, then $I_G \times I_H$ is independent
in $G\,\Box\,H$ and
\[G\,\Box\,H - N_{G\,\Box\,H}[I_G \times I_H]=(G-I_G)\,\Box\,(H-I_H)\,.\]
If in addition $G\,\Box\,H$ is well-dominated, then by Observation~\ref{obs:prelimfacts} it follows that
$(G-I_G)\,\Box\,(H-I_H)$ is well-dominated.
\end{obs}

\medskip

\noindent \textbf{Theorem~\ref{thm:CartesianWD}} \emph{
If $G\,\Box\, H$ is connected and well-dominated, then  $G$ or $H$ is a complete graph.
}

\begin{proof}
Suppose to the contrary that there exists a connected, well-dominated Cartesian product $G\,\Box\, H$ such that neither $G$ nor $H$ is a complete graph.  By Observation~\ref{obs:connectedCP}, both $G$ and $H$ are connected.  Among all such counterexamples, choose one where the order of $G\,\Box\, H$ is minimal. Since $G\,\Box\, H$ is also well-covered, it follows from Theorem~\ref{thm:wc-cartesian} that at least one of $G$ or $H$ is well-covered. Without loss of generality, we assume that $G$ is well-covered. Moreover, by Lemma~\ref{lem:path3}  we may assume that $G$ and $H$ each have order at least $4$.  Let $I_G$ be any maximal independent set of $G$ and let $I_H$ be any  maximal independent set of $H$.  By Observation~\ref{obs:wdreduction}, $(G-I_G)\,\Box\,(H-I_H)$ is well-dominated.  To simplify notation we will
let $G'=G-I_G$ and $H'=H-I_H$.
By Observation~\ref{obs:components} and the choice of $G \,\Box\, H$, the conclusion of the theorem holds for each component of $G'\,\Box\, H'$.
In particular, let $G_1, \ldots, G_{\ell}$ be the components of $G'$ and let $H_1,  \ldots, H_n$ be the components of $H'$. Thus, for each $(i,j) \in [\ell] \times [n]$,
$G_i$ or $H_j$ is a complete graph. Furthermore, by Theorem~\ref{thm:wdCart} this implies that if the orders of $G_i$ and $H_j$ are both at least $2$, then $G_i \,\Box\,  H_j$  is isomorphic to $K_m \,\Box\,  K_m$ for some $m \ge 2$ or to $P_3\,\Box\, K_3$.

We claim that we need not consider the case when all components of $H'$ are isolated vertices.  For suppose $I_H$ is a maximal independent set such that $H'=H-I_H$ consists
entirely of isolated vertices.  This implies that $H$ is bipartite.  Since $G$ is connected and $G$ and $H$ both have order at least $4$, it follows by Lemma~\ref{lem:bipartite} that $H$ is not a complete bipartite graph, which allows us to select a maximal independent set $J$ in $H$ so that $H - J$ contains an edge.  Thus we may assume $I_H$ has been chosen so that $H' = H- I_H$ contains a nontrivial component. Similarly, we may assume that the maximal independent set $I_G$ of $G$ has been chosen so that $G'=G-I_G$ contains a nontrivial component.

By Theorem~\ref{thm:wdCart} it follows that each component of $H'$ is either $K_1$, $P_3$, or $K_m$ for some $m \ge 2$.  The argument is now split into
four cases depending on the components of $H'$.

{\bf Case 1.} [{\bf $H'$ has $P_3$ as a component.}]
\vskip2mm
By Lemma~\ref{lem:path3}, every component of $G'$ is a complete graph of order $3$.  Furthermore, if $I$ is any maximal independent set of $G$,
then every component of $G-I$ is complete of order $3$.  For each $i \in [\ell]$ let $V(G_i)=\{x_1^i, x_2^i,x_3^i\}$.
Let $i \in [\ell]$.   Extend $\{x_1^i\}$ to a maximal independent set $I$ of $G$.  The edge $x_2^ix_3^i$ is in the graph $G-I$, which implies that there is a
vertex $v\in V(G-I)$ such that $\{vx_2^i,vx_3^i\} \subseteq E(G)$.  Note that $v\in I_G$, for otherwise the subgraph of $G$ induced by  $\{x_1^i,x_2^i,x_3^i,v\}$ is a subgraph of $G-I_G$.
Now extend $\{x_2^i\}$ to a maximal independent set $I'$ of $G$.  Since $\{x_1^ix_3^i,vx_3^i\} \subseteq E(G-I')$, we infer that $vx_1^i \in E(G)$.  Therefore,
for every $i \in [\ell]$, there exists $y_i \in I_G$ such that $\{y_i,x_1^i,x_2^i,x_3^i\}$ induces a complete graph of order $4$ in $G$.   Note that we have also proved here that if a vertex of $I_G$ is adjacent to two vertices of a triangle in $G$, then it is also adjacent to the other vertex of the triangle. (See the argument for $v$.)

We claim that $G = K_4$. To see this, note that since $G$ is connected, there exists a path from $x_1^1$ to $x_1^2$. Suppose first that for some $i \in \{2, \dots, \ell\}$, $j \in [3]$, and $k\in [3]$ that $x_j^1$ is adjacent to some vertex $z \in I_G - \{y_1\}$ and $z$ is also adjacent to $x_k^i$. Reindexing if necessary, we may assume $j=k=1$. Extend $\{x_2^1, x_1^i\}$ to a maximal independent set $I_1$ of $G$. Now the set $\{y_1, x_1^1, x_3^1, z\}$ belongs to a single component of $G- I_1$, which is a contradiction.
Therefore, if $w\in (I_G-\{y_1\}) \cap N(V(G_1))$, then $N(w) \subseteq V(G_1)$.
Next, suppose $y_1$ has a neighbor $x_k^i$ for some $k \in [3]$ and $i \in \{2, \dots, \ell\}$. Reindexing if necessary, we may assume $k=1$ and $i=2$. Extend $\{x_1^1, x_2^2\}$ to a maximal independent set $I_2$ of $G$. Now the set $\{x_2^1, x_3^1, y_1, x_1^2\}$ belongs to a single component of $G- I_2$, another contradiction. It follows that $N_G[y_1] = V(G_1)\cup \{y_1\}$ and $G= K_4$. This is a contradiction.

{\bf Case 2.}  [{\bf Some component of $H'$ is a complete graph of order $m$ for some $m \ge 4$.}]
\vskip2mm
Applying Theorem~\ref{thm:wdCart} we see that each component of $G'$ is either an isolated vertex or a complete graph of order $m$.  Furthermore, by our choice
of the maximal independent set $I_G$, at least one component of $G'$ is isomorphic to $K_m$.

Assume first that some component of $G'$ is $K_1$.   Therefore, we write $G' = G_1 \cup \cdots \cup G_{\alpha} \cup G_{\alpha+1} \cup \cdots \cup G_{\ell}$ where $G_i = K_1$ for $i \in [\alpha]$ and $G_i = K_m$ for $\alpha+1 \le i \le \ell$.  For each $i$  with $\alpha+1 \le i \le \ell$, let $V(G_i)= \{x_1^i, \ldots, x_m^i\}$ and let
$V(G_j)=\{x_1^j\}$ for $1 \le j \le \alpha$.  Let $J$ be the  maximal independent set of $G$ defined by $J = C \cup \bigcup_{i=1}^{\ell} \{x_1^i\}$ where $C = I_G - N_G[\cup_{i=1}^{\ell} \{x_1^i\}]$. Since $G$ is well-covered, $|J| = |I_G|$, which implies $\ell \le |I_G|$. Using the same argument as above $G'' = G - J$ is the disjoint union of  components  each of which is either $K_1$ or $K_m$.  Note that $I_G-C \subseteq V(G)-J=V(G'')$.  It follows that each vertex $x \in I_G - C$ is either an isolate in $G''$ or $x$ is adjacent to every vertex in $V(G_i) - \{x_1^i\}$ for some $\alpha+1 \le i \le \ell$ and $x$ is not adjacent to any vertex of $V(G_j) - \{x_1^j\}$ for every $j \in \{\alpha+1, \ldots, \ell\}-\{i\}$.

Let $\{Y_G, Z_G, M_G\}$ be the weak partition of $I_G$ defined as follows.
\begin{itemize}
\item $Y_G$ is the set of all $g \in I_G$ such that $V(G_i) \cup \{g\}$ induces in $G$ a complete graph (of order $m+1$) for some $i$ with $\alpha+1 \le i \le \ell$,
\item $Z_G$ is the set of all $g \in I_G$ that are isolated in $G''$, and
\item $M_G = I_G-(Y_G \cup Z_G)$.
\end{itemize}
First we claim that $|Y_G|=\ell-\alpha$.  Let $i$ be an index with $\alpha+1 \le i \le \ell$.  Since each component of $G''$ is either $K_1$ or
$K_m$ and the subgraph of $G''$ induced by $V(G_{i})-\{x_1^{i}\}$ is a complete graph of order $m-1$, we see that there exists a vertex
$y \in I_G$  that is adjacent to every vertex in $V(G_{i})-\{x_1^{i}\}$ and is not adjacent to any other vertex in $G''$.
Suppose $y$ is not adjacent to $x_1^{i}$.  Extend $\{x_2^{i}\}$ to a maximal independent set $I$ of $G$.
In $G-I$ the set $(V(G_{i})-\{x_2^{i}\}) \cup \{y\}$ induces a subgraph of order $m$ that is not a complete graph.  This is a contradiction,
and thus $y \in Y_G$.   This also shows that  $|Y_G| \ge \ell-\alpha$.  Furthermore, suppose there exist distinct vertices $u$ and $v$ in $Y_G$
such that $V(G_j) \subseteq N_G(u) \cap N_G(v)$ for some $\alpha+1 \le j \le \ell$.
By enlarging $\{x_1^j\}$ to a maximal independent set $I'$ of $G$, we see that $u$ and $v$ belong to a component of order $m+1$ of $G-I'$.  This contradiction shows
that $|Y_G| \le \ell-\alpha$.  We thus denote $Y_G$ as $\{y_1, \ldots,y_t\}$, where $t=\ell-\alpha$ and such that $V(G_{\alpha +i}) \subseteq N_G(y_i)$
for each $1 \le i \le t$.

Next, we claim that $M_G\subseteq C$. To see this, suppose $x \in M_G$.  Then, $x$ is not isolated in $G''$, and it follows that $x$ is adjacent to a vertex $x_j^i$  for some $2 \le j \le m$ and some $\alpha+1\le i \le \ell$ since $G$ is connected. Reindexing if necessary, we may assume $x$ is adjacent to $x_2^{\alpha+1}$. Since $x \not\in Y_G$, there exists $k \in [m]-\{2\}$ such that
$xx_k^{\alpha+1}\not\in E(G)$. Suppose $x$ is adjacent to a vertex in $V(G_{\alpha+1}) - \{x_2^{\alpha+1}\}$. When we extend $\{x_2^{\alpha+1}\}$ to a maximal independent set $I$ of $G$, the set $(V(G_{\alpha+1}) - \{x_2^{\alpha+1}\}) \cup \{x\}$ is contained in a component that is complete, which implies $x$ is adjacent to every vertex of $V(G_{\alpha+1})$, a contradiction. Thus, $x$ is not adjacent to $x_1^{\alpha+1}$. Moreover, if $x$ is adjacent to
$x_1^m$ for some $m \in [\ell]-\{\alpha+1\}$, then when we extend $\{x_3^{\alpha+1}, x_1^m\}$ to a maximal independent set $I$, there exists a component in $G-I$ containing $\{x\} \cup (V(G_{\alpha+1}) - \{x_3^{\alpha+1}\})$, which is a contradiction. Thus, $x \in I_G - N_G[\cup_{i=1}^{\ell} \{x_1^i\}] = C$.

Let $Q_2 = \left(\bigcup_{i=1}^{\alpha}\{x_1^i\}\right) \cup \left( \bigcup_{i=\alpha+1}^{\ell}\{x_2^i\}\right)$ and let $M_2 = I_G - N_G[Q_2]$.  The set $Q_2$ is independent and the set $T_2$
defined by $T_2 = M_2 \cup Q_2$ is a maximal independent set of $G$.  Let $G''' = G - T_2$.  Recall that  $H'$ has a component isomorphic to $K_m$.
Using the same argument as above, we see that each component of $G'''$ is a complete graph of order $1$ or $m$, and at least one of these components has order $m$
since $x_1^{\alpha+1}x_3^{\alpha+1}$ is an edge in $G'''$.

Suppose $Z_G \neq \emptyset$.  Let $z \in Z_G$.
Since $z$ is an isolate in $G''$, it follows that $z$ is not adjacent in $G$  to any vertex in $V(G_k)-\{x_1^k\}$, for $\alpha+1 \le k \le \ell$.
Also, $z$ is not adjacent in $G$  to $x_1^k$ for any $k$ with $\alpha+1 \le k \le \ell$ since $z$ is not adjacent to $x_3^k$, for otherwise
$z$ is contained in a component of $G'''$  that is not a complete graph of order $m$.
Thus, $N_G(z) \subseteq \{x_1^1, \ldots,x_1^{\alpha}\}$.

Since $G$ is connected, there is a path in $G$ from $z$ to $y_i$ for each $i\in[t]$.  Without loss of generality we may assume that a shortest such
path, say $P$, is a $z,y_1$-path.  Recall, as we proved above, that $M_G \subseteq C$, which implies that $N_G(v) \cap (\cup_{i=1}^{\ell} \{x_1^i\})=\emptyset$
for every $v \in M_G$.  Recall also that $N_G(z) \subseteq \{x_1^1, \ldots,x_1^{\alpha}\}$.  From this we infer that
\[V(P) \subseteq \{y_1\} \cup Z_G \cup \{x_1^1, \ldots,x_1^{\alpha}\}\,.\]
Reindexing $\{x_1^1,\ldots,x_1^{\alpha}\}$  if necessary, we may assume
without loss of generality that $y_1x_1^1$ and $z'x_1^1$ are both edges of $G$ for some $z'\in Z_G$. Since $N_G(z') \subseteq \{x_1^1, \ldots,x_1^{\alpha}\}$,
expand the independent set $\{z', x_2^{\alpha+1}\}$ to a maximal independent set $I$ of $G$.
Now we see that $\{x_1^1, y_1\} \cup (V(G_{\alpha+1})-\{x_2^{\alpha+1}\})$ does not induce a complete graph of order $m$ in $G-I$.
This is a contradiction, and therefore $Z_G = \emptyset$.  As a result
\[|I_G| = |Y_G| + |M_G| \le \ell + |M_G| \le \ell + |C| = |J|=|I_G|.\]
This implies that $|Y_G|=\ell$, which in turn implies that $G'$ has no isolated components.  This contradicts our assumption.  Hence, every component of $G'$
is a complete graph of order $m$.

Since $G$ is well-covered,
\[ |M_2| + \ell = |M_2| + |Q_2| = |T_2| = |I_G| = |M_G| + \ell.\]
Thus, $|M_2| = |M_G|$. Moreover, $M_2 \subseteq M_G$ as each vertex $w$ in $M_2$ is not adjacent to any vertex of the form $x_2^k$ for
$k \in [\ell]$ and so $w \notin Y_G$.
It follows that $M_G = M_2$. In fact, if for $3 \le j \le m$ we define
\[Q_j = \bigcup_{i=1}^{\ell} \{x_j^i\} \hskip15mm \text{ and } \hskip15mm M_j = I_G - N_G[Q_j],\]
then $T_j = M_j \cup Q_j$ is a maximal independent set of $G$ from which - as we did above - we may conclude that $M_j = M_G$.
That is, $M_G$  consists of vertices which are not adjacent to any vertex of $G_1 \cup \cdots \cup G_{\ell}$. Since $G$ is connected, it follows that
$M_G= \emptyset$.  Moreover, we know from above that $y_i$ cannot have a neighbor in $G_{\alpha+j}$ when $i \ne j$.
Therefore, $I_G$ consists of a single vertex adjacent to every vertex in $G'$.  That is, $G = K_{m+1}$, which contradicts the choice of $G \,\Box\,  H$.

{\bf Case 3}. [{\bf Some component of $H'$ is $K_3$.}]
\vskip2mm

We may assume that for every maximal independent set $I$ of $H$ the graph $H- I$ does not contain a component isomorphic to a path of order $3$ or a complete graph of order at least $4$,  for otherwise we are in Case 1 or Case 2.  Since $H'$ contains a $K_3$,  it follows from
Theorem~\ref{thm:wdCart} that every component of $G'$ is either $K_1$, $K_3$, or $P_3$.  Furthermore, by the choice of $I_G$,  at least one component of $G'$  nontrivial.  By Theorem~\ref{thm:wdCart} we infer that no component of $H'$ is $K_2$.

{\bf Subcase 3.1}. [{\bf $H'$ contains an isolated vertex.}]
\vskip2mm

We write $H' = H_1 \cup \cdots \cup H_{\sigma} \cup H_{\sigma +1} \cup \cdots \cup H_n$ where $H_i = K_1$ for $i \in [\sigma]$ and $H_i = K_3$ for $\sigma+1 \le i \le n$.
For each $i$  with $\sigma+1 \le i \le n$, let $V(H_i)= \{x_1^i, x_2^i,  x_3^i\}$ and let
$V(H_j)=\{x_1^j\}$ for $1 \le j \le \sigma$.  The set $J$ defined by  $J = C \cup \bigcup_{i=1}^{n} \{x_1^i\}$, where $C = I_H - N_H[\cup_{i=1}^{n} \{x_1^i\}]$,
is a maximal independent set in $H$. Using the same argument as above, the components of  $H'' = H - J$ are each complete graphs of order $1$ or $3$ and at least one of these components has order $3$ since $x_2^{\sigma+1}x_3^{\sigma+1}$ is an edge in $H''$.  It follows that each vertex $u \in I_H - C$ is either an isolate in $H''$ or $u$ is adjacent to every vertex in $V(H_i) - \{x_1^i\}$ for some $\sigma+1 \le i \le n$ and $u$ is not adjacent to any vertex of $V(H_j) - \{x_1^j\}$ for every $j \in \{\sigma+1, \ldots, n\}-\{i\}$.
\vskip2mm

Let $\{Y_H, Z_H, M_H\}$ be the weak partition of $I_H$ defined as follows.
\begin{itemize}
\item $Y_H$ is the set of vertices $h$ in $I_H$ such that $V(H_i) \cup \{h\}$ induces a complete graph of order $4$ in $H$ for some $\sigma+1 \le i \le n$,
\item $Z_H$ is the set of vertices in $I_H$ that are isolated in $H''$, and
\item $M_H=I_H-(Y_H \cup Z_H)$.
\end{itemize}

First we claim that $Y_H \ne \emptyset$. To see this, recall that $H''$ has a component isomorphic to $K_3$ and every nontrivial
component of $H''$ is a complete graph of order $3$.  For each $i \in [n]-[\sigma]$, the vertices
$x_2^i$ and $x_3^{i}$ are adjacent in $H''$, and thus there exists a vertex $y \in I_H$  that is adjacent to both $x_2^{i}$ and $x_3^{i}$.
Suppose $y$ is not adjacent to $x_1^{i}$.  Extend $\{x_2^{i}\}$ to a maximal independent set $I$ of $H$.
In $H-I$ the set $\{y, x_1^{i},x_3^{i} \}$ induces a path of order $3$.  This is a contradiction,
and the claim is established. Moreover, this shows that for each $\sigma+1 \le i \le n$, there exists $y \in I_H$ such that $V(H_i)\cup \{y\}$ induces a complete graph of order $4$ in $H$. Hence, $|Y_H| \ge n - \sigma$.  By definition  every vertex in $Y_H$ is in $H''$.
Since each component of $H''$ is a complete graph of order $3$ or an isolated vertex, $|Y_H|=n-\sigma$.
We thus denote $Y_H$ as $\{y_1, \ldots,y_t\}$, where $t=n-\sigma$ and such that $V(H_{\sigma +i}) \subseteq N_H(y_i)$ for each $1 \le i \le t$.

Let $Q_2 = \left(\bigcup_{i=1}^{\sigma}\{x_1^i\}\right) \cup \left( \bigcup_{i=\sigma+1}^n\{x_2^i\}\right)$ and let $M_2 = I_H - N_H[Q_2]$.
The set $Q_2$ is independent and the set $T_2$  defined by $T_2 = M_2 \cup Q_2$ is a maximal independent set in $H$.  Let $H''' = H - T_2$.
Recall that  $G'$ has a component isomorphic to $K_3$ or $P_3$.
Using the same argument as above, we see that each component of $H'''$ is a complete graph of order $1$ or $3$, and at least one of these components is isomorphic to $K_3$.

We claim that $N_H(M_H) \cap \{x^1_1, \ldots,x^{\sigma}_1\}=\emptyset$.  Suppose this is not the case, and let $x \in M_H$ such that
$x$ has a neighbor, say $x_1^s$, where $s\in [\sigma]$.  Since $x$ is not isolated in $H''$, there exists a neighbor of $x$ in
$\cup_{i=\sigma+1}^n \{x^i_2,x^i_3\}$.  Without loss of generality we assume $xx^{\sigma+1}_2 \in E(H)$.  Since $x \notin Y_H$, there exists $k \in \{1,3\}$
such that $xx^{\sigma+1}_k \notin E(H)$. Let $I'$ be a maximal independent set of $H$ that contains $\{x_1^s,x^{\sigma+1}_k\}$.
Then $\{x,y_1\} \cup (V(G_{\sigma+1})-\{x^{\sigma+1}_k\})$ belongs to a component of order at least $4$ in $H-I'$, which is a contradiction.
Therefore, $N_H(M_H) \cap \{x^1_1, \ldots,x^{\sigma}_1\}=\emptyset$.

Suppose $Z_H \neq \emptyset$.  Let $z \in Z_H$.
Since $z$ is an isolate in $H''$, it follows that $z$ is not adjacent in $H$  to any vertex in $V(H_k)-\{x_1^k\}$, for $\sigma+1 \le k \le n$.
Also, $z$ is not adjacent in $H$  to $x_1^k$ for any $k$ with $\sigma+1 \le k \le n$ since $z$ is not adjacent to $x_3^k$, for otherwise
$z$ is contained in a component of $H'''$  that is not a complete graph of order $3$.
Thus, $N_H(z) \subseteq \{x_1^1, \ldots,x_1^{\sigma}\}$.

Since $H$ is connected, there is a path in $H$ from $z$ to $y_i$ for each $i\in[t]$.   Without loss of generality we may assume that a shortest such
path, say $P$, is a $z,y_1$-path.  Recall from above that $N_H(M_H) \cap \{x^1_1, \ldots,x^{\sigma}_1\}=\emptyset$.  Recall also that
$N_H(z) \subseteq \{x_1^1, \ldots,x_1^{\sigma}\}$.  From this we infer that
\[V(P) \subseteq \{y_1\} \cup Z_H \cup \{x_1^1, \ldots,x_1^{\sigma}\}\,.\]
 Reindexing $\{x_1^1,\ldots,x_1^{\sigma}\}$  if necessary, we may assume
without loss of generality that $y_1x_1^1$ and $z'x_1^1$ are both edges of $H$ for some $z'\in  Z_H$. Since $z'$ is not adjacent to all the
vertices of $H_{\sigma+1}$ we can expand $\{z', x_j^{\sigma+1}\}$ to a maximal independent set $I$ of $H$ for some $x_j^{\sigma+1} \notin N_H(z')$.
We now infer that $\{x_1^1, y_1\} \cup (V(H_{\sigma+1}) -\{x_j^{\sigma+1}\})$ induces a subgraph of order $4$ in $H-I$.  This is a contradiction,
and therefore $Z_H = \emptyset$.

We claim that if $M_H \neq \emptyset$, then each vertex in $M_H$ has exactly one neighbor in $H$ and that neighbor belongs to
$\cup_{i=\sigma+1}^n \{x_2^i,x_3^i\}$.  Suppose there exists $h \in M_H$.  Recall that $N_H(h) \cap \{x^1_1, \ldots,x^{\sigma}_1\}=\emptyset$.
Since $h$ is not isolated in $H''$ and $H$ is connected, we see that $h$ has a neighbor in
$\cup_{i=\sigma+1}^n \{x_2^i,x_3^i\}$.  We may assume without loss of generality that $hx_2^{\sigma+1} \in E(H)$.
If $h$ has another neighbor in $H_{\sigma+1}$, by choosing a maximal
independent set $I$ of $H$ containing  $x_2^{\sigma+1}$ it follows that $H-I$ has a component of order at least $4$ that contains
$\{y_1, h, x_1^{\sigma+1}, x_3^{\sigma+1}\}$.  This is a contradiction. Therefore, $h$ has at most one neighbor in $H_i$ for
each $i \in [n]-[\sigma]$.  Suppose next that $h$
has a neighbor $w \in V(H_j)$ for some $j$ with $\sigma+1 < j \le n$.  Let $I$ be a maximal independent set in $H$ such that $\{w, x_1^{\sigma+1}\} \subseteq I$.
Now the set $\{h,y_1,x_2^{\sigma+1}, x_3^{\sigma+1}\}$ induces a connected subgraph of $H-I$.  This is a contradiction.  Therefore, if $h \in M_H$, then
$\deg_H(h)=1$ and the only neighbor of $h$ is in the set $\cup_{i=\sigma+1}^n \{x_2^i,x_3^i\}$.

Suppose some vertex in $Y_H$, say $y_1$, is adjacent in $H$ to a vertex that does not belong to $V(H_{\sigma+1})$.
Let $L=\{u \in V(H')-V(H_{\sigma+1})\, :\, u \in N_H(y_1)\}$.  From what was shown earlier we know that $L \subseteq \cup_{i=1}^{\sigma}\{x_1^i\}$.
Suppose  there exists a vertex, say  $x_1^k \in L$, such that $\deg_H(x_1^k) \ge 2$.  Since no vertex of $M_H$ is adjacent to $x_1^k$ by the claim in
the paragraph above, let $y_j \in N_H(x_1^k)-\{y_1\}$.
Enlarge $\{y_j,x_1^{\sigma+1}\}$ to a maximal independent set $I$ of $H$.  Then $\{x_1^k,y_1,x_2^{\sigma+1},x_3^{\sigma+1}\}$ induces a connected subgraph in $H-I$,
which is again a contradiction, and it follows that every vertex in $L$ is a leaf in $H$.  Since $H$ is connected, this implies that $H\in \cF_2$,
and thus $G \,\Box\, H$ is not well-dominated by Lemma~\ref{lem:spwd}.  This is a contradiction.

{\bf Subcase 3.2}. [{\bf Every component of $H'$ is a complete graph of order $3$.}]
\vskip2mm

Note that from above we have each vertex $u\in I_H -C$ that is adjacent to every vertex of $V(H_i) - \{x_1^i\}$ for $i\in [n]$ is not adjacent to any vertex of $V(H_j) - \{x_1^j\}$ for $i\ne j$. Moreover, our arguments that $Y_H \ne \emptyset$ and $Z_H=\emptyset$ did not rely on $H'$ containing isolates. Therefore, we may assume that there exists $y_1 \in Y_H$ such that $N(y_1)= V(H_1)$. Furthermore, our argument that if $M_H \ne \emptyset$ implies $M_H$ has exactly one neighbor in $H$ that belongs to $\cup_{i=\sigma+1}^n \{x_2^i,x_3^i\}$ did not rely on $H'$ containing isolates. However, since $H$ is connected, this leads us to $H \in \cF_2$ or
$H=K_4$, both of which are contradictions, one by Lemma~\ref{lem:spwd} and one by our choice of $G \,\Box\, H$.

\vskip2mm
{\bf Case 4.} [{\bf Some component of $H'$ is $K_2$.}]
\vskip2mm
We may assume that each component of $H'$ is a complete graph of order $1$ or $2$ since the other possibilities have been argued in the first three cases.
Since $G'\,\Box\, H'$ is well-dominated, it follows from Theorem~\ref{thm:wdCart} that each component of $G'$ is $K_1$ or $K_2$. Recall that we may also assume that $G'$ has at least one component that is a complete graph of order $2$.

Assume that $G'$ contains an isolated vertex.
Let $G' = G_1 \cup \cdots \cup G_{\alpha} \cup G_{\alpha+1} \cup \cdots \cup G_{\ell}$ where $G_i = K_1$ for $i \in [\alpha]$ and $G_i = K_2$ for $\alpha+1 \le i \le \ell$.  For $\alpha+1 \le i \le \ell$, let $V(G_i)= \{x_1^i,x_2^i\}$  and for each $G_i = K_1$, label the vertex $x_1^i$. Let $J$ be the  maximal independent set of $G$ defined by
$J = C \cup\, \bigcup_{i=1}^{\ell} \{x_1^i\}$, where $C = I_G - N_G[\cup_{i=1}^{\ell} \{x_1^i\}]$.  Using the same argument as above, $G'' = G - J$ is the disjoint union of  components  each of which is either $K_1$ or $K_2$. It follows that each vertex $x \in I_G - C$ is either an isolate in $G''$ or is adjacent to a vertex  $x_2^i$ for some $\alpha+1 \le i \le \ell$. Note that if $x \in I_G-C$, then $x$ has at most one neighbor in the set $\cup_{i=\alpha+1}^{\ell}\{x_2^i\}$,
for otherwise $x$ belongs to a component of order at least $3$ in $G''$.

\vskip2mm

Suppose $G$ contains a triangle.  This means there exists a vertex $g \in I_G-C$ and an index $k\in [\ell]-[\alpha]$ such that $\{g,x_1^k,x_2^k\}$ induces a
triangle in $G$.  Without loss of generality we may assume $k=\alpha+1$.  Suppose $x_1^{\alpha+1}$ has a neighbor $w \in I_G - \{g\}$ that is not a leaf in $G$. If $w$ is adjacent to $x_2^{\alpha+1}$, then we can pick a maximal independent set $I$ of $G$ containing $x_2^{\alpha+1}$. However, $G-I$ contains the path induced by $\{w,g,x_1^{\alpha+1}\}$, which is a contradiction. If $w$ is adjacent to some $x_s^j$ for $j \ne \alpha+1$ and $s \in \{1, 2\}$, then we can pick a maximal independent set $I$ containing $x_s^j$ and $x_2^{\alpha+1}$ and again we have a contradiction. Therefore, $N_G(x_1^{\alpha+1}) - \{g, x_2^{\alpha+1}\}$ is either empty or consists of leaves in $G$. A similar argument can be used to show that $N_G(x_2^{\alpha+1}) - \{g, x_1^{\alpha+1}\}$  is either empty or consists of leaves in $G$.
Finally, suppose that $g$ is adjacent to some $x_r^j$ where $j \ne \alpha+1$, $r \in \{1, 2\}$ and $x_r^j$ is not a leaf in $G$. Since no vertex in $I_G-C$ has more than
one neighbor in $\cup_{i=\alpha+1}^{\ell}\{x_2^i\}$, we infer that $r=1$.  Also, $j \notin [\ell]-[\alpha]$ for otherwise $\{g, x_1^{\alpha+1}, x_1^j\}$ induces a
path in $G-I$ where $I$ is a maximal independent set of $G$ that contains $\{x_2^j, x_2^{\alpha+1}\}$.
 Since $x_1^j$ is not a leaf in $G$, we have $zx_1^j \in E(G)$ for some vertex $z \in I_G$.  It follows
that $z$ is not adjacent to $x_2^{\alpha+1}$ since $N_G(x_2^{\alpha+1}) - \{g, x_1^{\alpha+1}\}$ consists of leaves in $G$.  Therefore, we can choose a maximal independent set $I$ containing $\{z, x_2^{\alpha+1}\}$ and now $G- I$ contains the path of order $3$ induced by  $\{x_1^j, g, x_1^{\alpha+1}\}$, which is a contradiction.  Therefore,
$N_G(g)-\{x_1^{\alpha+1},x_2^{\alpha+1}\} $ is either empty or consists entirely of leaves in $G$.  Combining these conclusions about the neighbors of
$g, x_1^{\alpha+1}$ and $x_2^{\alpha+1}$  and the fact that $G$ is connected, we see that $G \in \cF_1$ or $G=K_3$. If $G \in \cF_1$, it follows from Lemma~\ref{lem:spwd} that
$G \,\Box\,  H$ is not well-dominated. On the other hand, $G \ne K_3$ as we have assumed $G$ has order at least $4$.

Therefore, $G$ does not contain a triangle. Furthermore, the same argument can be used to prove that $H$ does not contain a triangle. Hence, $G$ and $H$
are both of order at least $3$ and  girth at least $4$.  By  Lemma~\ref{lem:wcgirth4}, $G\,\Box\,  H$ is not well-covered and therefore not well-dominated.  This is a contradiction.
Hence we therefore assume that $G'$ does not have an isolated vertex.

Now, suppose a vertex $g \in I_G-C$ belongs to a triangle in $G$
induced by $\{g,x_1^k,x_2^k\}$ for some $k\in [\ell]$.  Using the argument in the paragraph above it follows that the two sets
$N_G(x_1^k) - \{g, x_2^k\}$ and $N_G(x_2^k) - \{g, x_1^k\}$  are either empty or consist of leaves in $G$.  On the other hand, $N_G(g)-\{x_1^k,x_2^k\}=\emptyset$ since
$G'$ does not contain an isolated vertex.  Once again it follows that $G \in \cF_1$ or $G=K_3$, both of which lead to a contradiction.  Therefore $G$ does not contain
a triangle and as above we infer by Lemma~\ref{lem:wcgirth4} that $G\,\Box\,  H$ is not well-covered and therefore not well-dominated. This is a contradiction.

We have shown that every possible case concerning the components of $H'$ leads to a contradiction. This finishes the proof.
\end{proof}

Using Theorem~\ref{thm:wdCart} together with Theorem~\ref{thm:CartesianWD} we now have a complete characterization of nontrivial, connected Cartesian products that
are well-dominated.
\begin{corollary}
A nontrivial, connected Cartesian product $G\,\Box\,H$ is well-dominated if and only if $G\,\Box\,H = P_3 \,\Box\,K_3$  or $G\,\Box\,H= K_n \,\Box\,K_n$ for some $n\ge 2$.
\end{corollary}

\section*{Acknowledgements}
We would like to thank Erika King and Michael O'Grady for finding the error in our original proof to Theorem~\ref{thm:wdCartesian} in \cite{akr-2021}.
We also thank the referees for a number of helpful suggestions and for helping us to clarify some of the proofs in this paper.

\end{document}